\documentclass [12pt]{article}
\usepackage{amssymb}
\usepackage{amsmath}
\usepackage{epsfig}
\usepackage{graphicx}
\usepackage{amsthm}
\usepackage{amssymb}
\usepackage{amsfonts}
\usepackage{dirtytalk}
\allowdisplaybreaks
\usepackage{enumerate}
\usepackage{enumerate}
\newtheorem{theorem}{Theorem}
\newtheorem{lemma}{Lemma}

\begin{document}
\thispagestyle{empty}
\null\vspace{-1cm}
\medskip
\vspace{1.75cm}
\centerline{\textbf{{Bounds for the Zeros of Polynomials over Quaternion Division Algebra}}}
~~~~~~~~~~~~~~~~~~~~~~~~~~~~~~~~~~~~~~~~~~~~~~~~~~~~~~~~~~~~~~~~~~~~~~~~~~~~~~~~~~~~~~~~~~~~~~~~~~~~~~~~~~~~~~~~~~~~~~~~~~~~~~~~~~~~~~~~~~~~~~~~~~~~~~~~~~~

\centerline{\bf  {Ovaisa Jan} and  {Idrees Qasim}}
\centerline {Department of Mathematics, National Institute of Technology, Srinagar, India-190006}
\centerline { ovaisa\_2022phamth009@nitsri.ac.in, idreesf3@nitsri.ac.in}
~~~~~~~~~~~~~~~~~~~~~~~~~~~~~~~~~~~~~~~~~~~~~~~~~~~~~~~~~~~~~~~~~~~~~~~~~~~~~~~~~~~~~~~~~~~~~~~~~~~~~~~~~~~~~~~~~~~~~~~~~~~~~~~~~~~~~~~~~~~~~~~~~~~~~~~~~
\vskip0.1in
\textbf{Abstract:} Locating the zeros of quaternionic polynomials is a fundamental problem with significant implications across scientific and engineering disciplines, yet the noncommutative nature of quaternion multiplication makes it fundamentally more complex than the classical complex case. In this paper, we develop new bounds for the zeros of polynomials with quaternionic coefficients. 
  We establish spectral norm inequalities for quaternionic matrices, particularly those of a partitioned form. These inequalities are applied to specialized quaternionic companion matrices to derive novel upper bounds for the zeros of the original polynomial. 
By establishing novel spectral norm inequalities for partitioned quaternionic matrices and utilizing the structural properties of companion matrices and their higher powers, we derive unexplored upper bounds for the zeros of quaternionic polynomials. Our bounds are systematically sharper than existing results and provide a unified framework for zero localization in the quaternionic setting. \\

	\vspace{0.10in}
	~~~~~~~~~~~~~~~~~~~~~~~~~~~~~~~~~~~~~~~~~~~~~~~~~~~~~~~~~~~~~~~~~~~~~~~~~~~~~~~~~~~~~~~~~~~~~~~~~~~~~~~~~~~~~~~~~~~~~~~~~~~~~~~~~~~~~~~~~~~~~~~~~~~~~~~~~
	
	\noindent {{\bf Mathematics Subject Classiﬁcation (2020):} Primary 46L55, Secondary 44B20.}\\
	\vspace{0.1in}

\section{Introduction:} In recent years, determining the zeros of quaternionic polynomials and establishing rigorous bounds for these polynomials have gained much attention of the researchers because of their wide-ranging applications in both pure and applied sciences, including quantum physics, control theory, computer graphics, and signal processing (see, for instance, \cite{ASL,ASS,CJH,HTL} and the references therein). The study of locating and estimating zeros of quaternionic polynomials traces back to Niven \cite{IN}. Subsequent algorithmic developments were provided by Serodio et al. \cite{SR}, and further investigations into structural properties and computational methods can be found in \cite{GM,JD,O,PS}.\\ \indent Nevertheless, systematic zero localization for quaternionic polynomials remains less advanced compared to the classical complex setting. Initial bounds were established by Janovská and  Opfer \cite{JD,O} while Kalantari \cite{BK} later obtained norm bounds for unilateral quaternionic polynomials. Some recent developments on the location and computation of zeros of quaternionic polynomials can be found in  \cite{ASS,JO,IQ}. Besides, the location of zeros of these polynomials the location of eigenvalues of quaternion polynomials is very well studied. Unlike complex matrices, quaternions admit distinct left eigenvalues and right eigenvalues. Quaternionic matrix have been extensively studied, with particular focus on the existence and location of their left and right eigenvalues. The existence of the left and right eigenvalues for quaternionic matrices has been given in  \cite{BA,BJL,WRM}. The location of the left and right eigenvalues for quaternionic matrices can be found in  \cite{ASL,BA,CJH,ABG,HTL,JQ,BK,PR,RL,WRL, ZF1,ZF2}.  
The intrinsic noncommutativity of quaternion multiplication significantly transforms the framework, resulting in a more complex spectral theory and requiring a rethinking of the foundational matrix analytic methods. We introduce a suitable framework for quaternionic companion matrices and establish extensions of key matrix inequalities for the spectral norm in the quaternionic setting. Since every monic polynomial is the characteristic polynomial of its Frobenius companion matrix. Consequently, the zeros of polynomials are precisely the eigenvalues of companion matrix . This isomorphism allows one to leverage the extensive machinery of matrix analysis.

  \section{Notation and Preliminary Knowledge:}  Throughout the paper the fields of real and complex numbers are denoted by $\mathbb{R}$ and $\mathbb{C}$ respectively. The set of quaternions is denoted by $\mathbb{H}$ and is defined as $$\mathbb{H}:=\{q=a+bi+cj+dk:a, b, c, d \in\mathbb{R} \},$$  
	 
	\noindent where $i^2=j^2=k^2=ijk=-1,~ij=k=-ji,~jk=i=-kj,~ki=j=-ik.$\\
 For $q \in\mathbb{H}$,~ ${\bar{q}}=a-bi-cj-dk$ is the conjugate of $q$ and hence the modulus of a quaternion $q$ is given by $|q|=\sqrt{a^{2}+b^{2}+c^{2}+d^{2}}.$ The set of \(n\)-column vectors with entries in \(\mathbb{H}\) is denoted by \(\mathbb{H}^n\). The inner product for \(x, y \in \mathbb{H}^n\) is \(\langle x, y \rangle := y^H x\), and the norm is \(\|x\| := \sqrt{\langle x, x \rangle}\).
  ${M}_{m\times n}{(\mathbb{H})}$ denotes the space of ${m\times n}$ quaternion matrices and ${M}_{n}{(\mathbb{H})}$ denotes the space of ${n\times n}$ matrices over $\mathbb{H}.$ For $A=[a_{ij}] \in{M}_{n}{(\mathbb{H})}$ the transpose of $A$ is  \(A^T = [a_{ji}]\) and the conjugate transpose of $A$ is \(A^H = \overline{(A^T)}\).  Next we define norms on matrix $A.$
  For a square matrix $A\in M_{n}(\mathbb{H})$
  with enteries $a_{ij}\in\mathbb{H}$, the following norms are defined:\\ 
 1. The 1-Norm: $\|A\|_{1}:=\max\limits_{1\le j \le n}\sum\limits_{i=1}^{n}|a_{ij}|=\|A^{H}\|_{\infty}$\\   
 2. The $\infty$-Norm: $\|A\|_{\infty}:=\max\limits_{1\le i \le n}\sum\limits_{j=1}^{n}|a_{ij}|=\|A^{H}\|_{1}$\\
 3. The 2-Norm (Spectral Norm): $\|A||_{2}:=\sup\limits_{x\neq 0}\left\{ \frac{\|Ax\|_{2}}{\|x\|_{2}}:~x\in\mathbb{H}^{n}                         \right\} =\|A^{H}\|_{2}$\\
 4. The Frobenius Norm: $\|A\|_{F}:={(~\mbox{trace}~   A^{H}A)}^{1/2}.$\\ 
 
\noindent Since quaternion multiplication is not commutative, there exist two types of eigenvalues for a quaternion matrix; left eigenvalues and right eigenvalues  which are defined as follows\\
 Let $A \in M_{n \times n}(\mathbb{H})$ denote the set of $n \times$ n matrices with quaternionic entries. A quaternion $\lambda \in \mathbb{H}$ is a right eigenvalue of A if there exists a non-zero vector $\mathbf{x}\ \in \mathbb{H}^n$  such that $A\mathbf{x}=\mathbf{x}\lambda$.
While as a quaternion $\lambda \in \mathbb{H}$ is a left eigenvalue of A if there exists a non-zero vector $\mathbf{x}\ \in \mathbb{H}^n$  such that: $A\mathbf{x}=\lambda\mathbf{x}$.
The set of all right eigenvalues of $A$ is denoted by $$\sigma_r(A)=\left\{\lambda \in \mathbb{H}: A\mathbf{x}=\mathbf{x}\lambda ~~~\mbox{for some non-zero}~~ \mathbf{x}\in \mathbb{H}^{n} \right\}$$ and the set of all left eigenvalues is denoted by $$\sigma_l(A)=\left\{\lambda \in \mathbb{H}: A\mathbf{x}=\lambda\mathbf{x} ~~~\mbox{for some non-zero}~~ \mathbf{x}\in \mathbb{H}^{n} \right\}.$$  The right and left spectral radii of a matrix $A$ are defined by
$$\rho_{r}(A) = \max \{ |\lambda| : \lambda \in \sigma_r(A) \}$$ and $$\rho_{l}(A) = \max \{ |\lambda| : \lambda \in \sigma_l(A) \}$$ respectively.\\ Next, we recall the definition of quaternionic polynomials.

\noindent\textbf{Quaternion Polynomial:}\\
A quaternion polynomial is a function \(f: \mathbb{H} \to \mathbb{H}\). Due to the non-commutativity of quaternion multiplication, such polynomials are classified into three distinct types: left, right and general.\\  
 $$q_l(z) = q_n z^n + q_{n-1} z^{n-1} + \cdots + q_0,~ ~\mbox{where} ~~ q_{i}, z \in\mathbb{H},~~ (0 \le i \le n)$$ is a left quaternion polynomial of degree $n$ whereas
 $$q_r(z) = z^n q_n + z^{n-1} q_{n-1} + \cdots + q_0 , 
~~\mbox{where}~~ q_{i}, z \in\mathbb{H},~~ (0 \le i \le n).$$ is right quaternion polynomial of degree $n.$ If the leading coefficient $q_{n}=1$, then the above polynomials are said to be monic. General polynomials are defined as \(q_G(z) = q_0 z q_1 z \cdots z q_n + \phi(z)\), where \(\phi(z)\) is a finite sum of lower-degree monomials of the form \(r_0 z^i r_1 \cdots z r_i\) (\(i < n\)). Unlike commutative polynomials, evaluation is side-sensitive, left-evaluation $q_l(z) = q_n z^n + q_{n-1} z^{n-1} + \cdots + q_0$ differs from right-evaluation $q_r(z) = z^n q_n + z^{n-1} q_{n-1} + \cdots + q_0 $ and factorization is inherently non-unique.\\ 
 \indent In this paper we shall take the quaternionic polynomials of the type $q_{r}(z)=\sum_{i=0}^{n}z^iq_i$. Two quaternionic polynomials of this type can be multiplied according to the convolution product (Cauchy multiplication rule): given $q_{r_{1}}(z)=\sum_{i=0}^{n}z^iq_i$ and $q_{r_{2}}(z)=\sum_{j=0}^{n}z^jt_j$ , we define
$$(q_{r_{1}}*q_{r_{2}})(z):=\sum_{i=0,1,\dots, n~~j=0,1,\dots,m}^{}z^{i+j}q_it_j.$$
If $q_{r_{1}}$ has real coefficients, then so called * multiplication coincides with the usual pointwise multiplication.
Please note that the * product in the quaternionic setting is associative but not, in general, commutative. This lack of commutativity results in a behavior of polynomials that is quite distinct from their behavior in the real or complex settings. For instance, a real or complex polynomial of degree $n$ can have at most $n$ (real or complex) zeros, counted with their multiplicity. However, in the quaternionic setting, the second-degree polynomial $q^2+1$ exhibits an infinite number of zeros.\\
For example, the polynomial \( q^2 + 1 \) factorizes infinitely many ways in $\mathbb{H}$, such as \( (q - i)(q + i) \), \( (q - j)(q + j) \), or \( (q - (i+j)/\sqrt{2})(q + (i+j)/\sqrt{2}) \), since its roots include all quaternions \( q \) satisfying \( q^2 = -1 \) (i.e., pure unit quaternions \( ai + bj + ck \) with \( a^2 + b^2 + c^2 = 1 \)). This non-uniqueness arises because roots form conjugacy classes. If \( q \) is a root, so is \( sqs^{-1} \) for any nonzero \( s \in \mathbb{H} \) reflecting the spherical geometry of solutions in \( \mathbb{R}^3 \).

The following result, found in \cite{GS}, presents a complete description of the zero sets of a regular product of two polynomials by relating them to the zero sets of the two individual factors.\\
\textbf{Theorem $A_1$.} Let $f$ and $g$ be given quaternionic polynomials. Then the convolution product $(f*g)(z_0)=0$ if and only if $f(z_0)=0$ or $f(z_0)\neq 0$ implies $g(f(z_0)^{-1}z_0f(z_0))=0$.\\

\noindent Let 
  $$f_l(z) =  z^n + q_{n} z^{n-1} + \cdots + q_{2}z+q_{1}$$ and
 $$f_r(z) = z^n  + z^{n-1} q_{n} + \cdots + zq_{2}+q_{1}  
$$ be a monic polynomial of degree $n\ge 2$ ${where}~~ q_{i}, z \in\mathbb{H},~~ (1 \le i \le n), q_{1}\ne 0.$
 \\ The  companion matrix for the monic quaternionic polynomial \(f_{r}(z)\)and \(f_{l}(z)\) are given by

\[
C_{f_r} = 
\begin{pmatrix}
0 & 0 & \cdots & 0 & -q_{1} \\
1 & 0 & \cdots & 0 & -q_{2} \\
0 & 1 & \cdots & 0 & -q_{3} \\
\vdots & \vdots & \ddots & \vdots & \vdots \\
0 & 0 & \cdots & 1 & -q_{n}
\end{pmatrix}
\]

and \[
C_{f_l} = 
\begin{pmatrix}
0 & 1 & 0 & \cdots & 0 & 0 \\
0 & 0 & 1& \cdots & 0 & 0 \\

\vdots & \vdots & \ddots & \vdots & \vdots \\
0 & 0 & 0 &  \cdots & 0 & 1\\
-q_{1}& -q_{2}& -q_{3} & \cdots & -q_{n-1} & -q_{n}\\
\end{pmatrix}.
\]

\noindent The spectral properties of companion matrices provide a powerful bridge between quaternionic polynomials and their zeros. It is well-known from (Proposition 1, \cite{SR}) that if $\lambda$ is a left eigenvalue of $C_{f_{l}}$, then $\lambda$ is a zero of $f_{l}(z)$ and if $\lambda$ is a left eigenvalue of $C_{f_{l}}$, then by (Corollary 1, {\cite{SR}}) it is also a right eigenvalue. 
Similarly, for the right companion matrix $C_{f_r}$, the left eigenvalues coincide exactly with the zeros of $f_r(z)$ \cite{SR}. Consequently, all zeros of $f_l(z)$ are the right eigenvalues of $C_{f_l}$. However, the converse does not hold in general, not every right eigenvalue of $C_{f_l}$ corresponds to a polynomial zero, as illustrated by specific counterexamples in \cite{KI}.

This subtle distinction between left and right eigenvalues underscores the need for careful spectral analysis when bounding polynomial zeros via companion matrices.

\section{ Bounds for the zeros of quaternionic polynomials}
Several bounds have been established for the zeros of quaternionic polynomials, typically expressed in terms of coefficient magnitudes. Dar et al. \cite{DR} gave proofs of some classical bounds including a bound due to Cauchy, which says that if $z$ is a zero of $f_{r},$ then \\
$$|z|\le 1+ \max_{0\le i \le n}|q_{i}|.$$
In this section, we derive  new bounds for the zeros of quaternionic polynomials by decomposing the companion matrix and its powers into structured parts whose norms can be estimated sharply. Throughout we consider the monic right quaternionic polynomial $f_r$. Suppose that
$$f_r(z) = z^n  + z^{n-1} q_{n} + \cdots + zq_{2}+q_{1}  
$$ be a monic polynomial of degree $n\ge 3,$ $\mbox{where}~~ q_{i}, z \in\mathbb{H},~~(1 \le i \le n)$. If $ q_{1}\ne 0$ then the companion matrix is nonsingular. In this paper, we assume that $q_{1}\ne 0.$
\\ The Frobenius companion matrix $C_{f_{r}}$ of $f_{r}$ is defined as\\
\[
C_{f_r} = 
\begin{pmatrix}
0 & 0 & \cdots & 0 & -q_{1} \\
1 & 0 & \cdots & 0 & -q_{2} \\
0 & 1 & \cdots & 0 & -q_{3} \\
\vdots & \vdots & \ddots & \vdots & \vdots \\
0 & 0 & \cdots & 1 & -q_{n}
\end{pmatrix}.
\]
We have
\[
C^{2}_{f_{r}} = 
\begin{pmatrix}
0 & 0 & 0 & \cdots &0& -q_{1}&q_{1}q_{n} \\
0 & 0 & 0 & \cdots &0& -q_{2} &q_{2}q_{n}-q_{1}\\
1 & 0 & 0 & \cdots &0& -q_{3} &q_{3}q_{n}-q_{2}\\
\vdots & \vdots & \vdots & \ddots & \vdots& \vdots&\vdots \\
0 & 0 & 0 & \cdots &0 & -q_{n-1} &q_{n-1}q_{n}-q_{n-2}\\
0 & 0 & 0 & \cdots &1 & -q_{n} & q^{2}_{n}-q_{n-1}
\end{pmatrix}
\] 
and
\[
C^{3}_{f_{r}} = 
\begin{pmatrix}
0 & 0 & 0 & \cdots &0& -q_{1}&q_{1}q_{n}&q_{1}q_{n-1}-q_{1}q^{2}_{n} \\
0 & 0 & 0 & \cdots &0& -q_{2} &q_{2}q_{n}-q_{1}&q_{2}q_{n-1}-q_{2}q^{2}_{n}+q_{1}q_{n}\\
0 & 0 & 0 & \cdots &0& -q_{3} &q_{3}q_{n}-q_{2}&q_{3}q_{n-1}-q_{3}q^{2}_{n}+q_{2}q_{n}-q_{1}\\
1&0&0&\cdots&0&-q_{4}&q_{4}q_{n}-q_{3}&q_{4}q_{n-1}-q_{4}q^{2}_{n}+q_{3}q_{n}-q_{2}\\
\vdots & \vdots & \vdots & \ddots & \vdots&\vdots&\vdots&\vdots \\
0 & 0 & 0 & \cdots &0 & -q_{n-1} &q_{n-1}q_{n}-q_{n-2}&q^{2}_{n-1}-q_{n-1}q^{2}_{n}+q_{n-2}q_{n}-q_{n-3}\\
0 & 0 & 0 & \cdots &1 & -q_{n} & q^{2}_{n}-q_{n-1}&q_{n}q_{n-1}-q^{3}_{n}+q_{n-1}q_{n}-q_{n-2}
\end{pmatrix}.
\] 
We employ matrix inequalities involving the spectral norm to the powers of  companion matrices which yield progressively tighter bounds for the zeros of $f_{r}.$\\

\noindent For the proof of the theorems we need the following lemmas. The following lemma is mentioned in \cite{ASS}
\begin{lemma}{\label{lemma 1}}
    Let $A= [a_{ij}]\in M_{n}(\mathbb{H)}.$ Then $\|A\|^2_{2}=\|A^{H}\|^{2}_{2}=\|A^{H}A\|_{2}=\|AA^{H}\|_{2}.$
\end{lemma}
\begin{lemma}{\label{lemma 2}}
    Let 
    \[F = 
\begin{bmatrix}
0 & 0 & \cdots & 0 & 0 \\
1 & 0 & \cdots & 0 & -q_{1} \\
0 & 1 & \cdots & 0 & -q_{2} \\
\vdots & \vdots & \ddots & \vdots & \vdots \\
0 & 0 & \cdots & 1 & -q_{n-1}
\end{bmatrix}.\]
Then $\|F\|^{2}_{2}=d+1$, where $d=\sum\limits_{j=1}^{n-1}|q_{j}|^{2}.$
\end{lemma}

\begin{proof}
    Decompose $F=M+N,$ where \\
    \[M = 
\begin{bmatrix}
0 & 0 & \cdots & 0 & 0 \\
0 & 0 & \cdots & 0 & -q_{1} \\
0 & 0 & \cdots & 0 & -q_{2} \\
\vdots & \vdots & \ddots & \vdots & \vdots \\
0 & 0 & \cdots & 0 & -q_{n-1}
\end{bmatrix}
,
\quad
N = 
\begin{bmatrix}
0 & 0 & \cdots & 0 & 0 \\
1 & 0 & \cdots & 0 & 0 \\
0 & 1 & \cdots & 0 & 0 \\
\vdots & \vdots & \ddots & \vdots & \vdots \\
0 & 0 & \cdots & 1 & 0
\end{bmatrix}\]
Then $MN^{H}=NM^{H}=0.$ So, by triangle inequality, we have 
\begin{eqnarray*}
    \|F\|^2_{2} &=&\|(M +N )^{H}(M +N)\|_{2}\\ &=& \|M^{H}M+N^{H}N\|_{2}\\&\le&  \|M^{H}M\|_{2} + \|N^{H}N\|_{2} \\&=&  \sum_{j=1}^{n-1}|q_{j}|^{2}+1.
\end{eqnarray*}
Hence the result follows.
\end{proof}
\noindent The next lemma can be found in \cite{ASS}.
\begin{lemma}
    If $\|.\|_{\beta}$ ($\beta=1,2,\infty, F),$ are the quaternionic matrix norms and if $A \in  M_{n}(\mathbb{H}).$ Then
    $\rho_{l}(A),\rho_{r}(A)\le \|A\|_{\beta}$.
\end{lemma}

\begin{theorem}
Let $f_r(z) = z^n  + z^{n-1} q_{n} + \cdots zq_{2}+ q_1 $ be a quaternion monic polynomial. Then the zeros of $f_r(z)$ satisfies the following inequality
\[
|z| \leq \alpha + \sqrt{1+\sum_{j=1}^{n-1} |q_j|^2 },
\] where $\alpha=\sqrt{\sum_{j=1}^{n} |q_j - q_{j-1}|^2}, ~q_{0}=0.$
\end{theorem}
\begin{proof} Consider the companion matrix $Cf_{r}$. Decompose it as $C_{f_r}=P+Q$ 
with

\[
P = \begin{pmatrix}
0 & 0 & 0 & \cdots & 0&0 \\
1 & 0 & 0 & \cdots &0& -q_{1} \\
\vdots & \vdots & \vdots & \ddots & \vdots &\vdots\\
0 & 0 & 0 & \cdots & 0&-q_{n-2} \\
0 & 0 & 0 & \cdots & 1&-q_{n-1}
\end{pmatrix}
,
\quad
Q = \begin{pmatrix}
0 & 0 & 0 & \cdots & 0&-q_{1} \\
0 & 0 & 0 & \cdots &0& -q_{2}+q_{1} \\
\vdots & \vdots & \vdots & \ddots & \vdots &\vdots\\
0 & 0 & 0 & \cdots & 0&-q_{n-1}+q_{n-2} \\
0 & 0 & 0 & \cdots & 0&-q_{n}+q_{n-1}
\end{pmatrix}.
\]These matrices satisfy $PQ^{H}=QP^{H}=0.$
 By the triangle inequality, we have  
 $\|C_{f_r}\|_{2
 }\le \|P\|_{2}+\|Q\|_{2}.$
 By using Lemma \ref{lemma 2}, we have \\
 $$\|P\|_{2}= 1+ \sqrt{\sum_{j=1}^{n-1} |q_j|^2 }$$ and  $$||Q||_{2}= \sqrt{\sum_{j=1}^{n} |q_j - q_{j-1}|^2}=\alpha,$$ where $q_{0}=0.$\\ Consequently
 $$\|C_{f_r}\|_{2}\le \alpha+ \sqrt{\sum_{j=1}^{n-1} |q_j|^2 + 1},$$ where $\sqrt{\sum_{j=1}^{n} |q_j - q_{j-1}|^2},~ q_{0}=0.$ Since each zero $z$ of $f_{r}(z)$ is an eigenvalue of $C_{f_r},$ we have $|z|\le \rho_{r}(C_{f_{r}})\le \|C_{f_{r}}\|_{2}.$ Therefore, 
 $|z|\le \|C_{f_r}\|_{2}$ for every zero $z$ of $f_r.$ 
\end{proof}

\begin{theorem}
    If $z$ is a zero of $f_{r},$ then\\
    $$|z|\le \left(1+\alpha+\beta+\sum_{j=1}^{n}|q_{j}|^{2}\right)^\frac{1}{6},$$
    where $\alpha=\sum_{j=1}^{n}|q_{j}q_{n-1}-q_{j}q^{2}_{n}+q_{j-1}q_{n}-q_{j-2}|^{2},~~\beta=\sum_{j=1}^{n}|q_{j}q_{n}-q_{j-1}|^{2},~ q_{-1}=q_{0}=0.$ 
\end{theorem}
\begin{proof}
    Consider $C_{f_{r}}^3$ and use a decomposition strategy similar to Theorem 2. \\Let
\[
C_{f_r}^3 = M + N + L + G,
\]
\[M = 
\begin{bmatrix}
0 & 0 & \cdots & 0&0 & q_{1}q_{n-1}-q_{1}q^{2}_{n} \\
0 & 0 & \cdots & 0&0& q_{2}q_{n-1}-q_{2}q^{2}_{n}+q_{1}q_{n}\\
0 & 0 & \cdots & 0 &0& q_{3}q_{n-1}-q_{3}q^{2}_{n}+q_{2}q_{n}-q_{1}\\
\vdots  &\vdots & \ddots &\vdots& \vdots & \vdots \\
0 & 0 & \cdots & 0 & 0&q^{2}_{n-1}-q_{n-1}q^{2}_{n}+q_{n-2}q_{n}-q_{n-3}\\
0 & 0 & \cdots & 0 & 0&q_{n}q_{n-1}-q^{3}_{n}+q_{n-1}q_{n}-q_{n-2}
\end{bmatrix},
\]
\[
N = 
\begin{bmatrix}
0 & 0 & \cdots & 0&q_{1}q_{n} & 0 \\
0 & 0 & \cdots & 0&q_{2}q_{n}-q_{1}& 0 \\
0 & 0 & \cdots & 0 &q_{3}q_{n}-q_{2}& 0 \\
\vdots  &\vdots & \ddots &\vdots& \vdots & \vdots \\
0 & 0 & \cdots & 0 & q_{n-1}q_{n}-q_{n-2}&0\\
0 & 0 & \cdots & 0 &q^{2}_{n}-q_{n-1}&0
\end{bmatrix}
\quad
L= 
\begin{bmatrix}
0 & 0 & \cdots  &-q_{1} & 0 &0\\
0 & 0 & \cdots  &-q_{2}& 0&0 \\
0 & 0 & \cdots  &-q_{3}& 0&0 \\
\vdots  &\vdots &  &\vdots& \vdots & \vdots \\
0 & 0 & \cdots  & -q_{n-1}&0&0\\
0 & 0 & \cdots  & -q_{n}&0&0
\end{bmatrix},
\]
\[
 G=\begin{bmatrix}
    0&0\\I_{n-3}&0
\end{bmatrix},\]
where $I_{n-3} $ is identity matrix of order $n-3.$
Then
\[
 MN^{H}=NM^{H}=ML^{H}= LM^{H} = NL^{H}=L N^{H}\]\[=MG^{H}=GM^{H}=NG^{H}=GN^{H} =LG^H=GL^{H}=0.
\]
Thus by Lemma \ref{lemma 2} and triangle inequality, we get\\

\[
\|C_{f_{r}}^3\|^2_{2} \leq 1 + \alpha+\beta + \sum_{j=1}^{n} |q_j|^2,
\]
 where 
$\alpha=\sum_{j=1}^{n}|q_{j}q_{n-1}-q_{j}q^{2}_{n}+q_{j-1}q_{n}-q_{j-2}|^{2},~~\beta=\sum_{j=1}^{n}|q_{j}q_{n}-q_{j-1}|^{2},~ q_{-1}=q_{0}=0.$
Therefore:
\[
\|C_{f_{r}}^3\|_{2} \leq \left( 1 + \alpha+\beta + \sum_{j=1}^{n} |q_j|^2\right)^{1/2}.
\]

\noindent For any zero $z$ of $f_{r}$ , we have $|z|^{3}\le \rho_{r}(C_{f_{r}}) \leq \|C_{f_{r}}^3\|_{2}$ for every zero  $z$ of $C_{f_{r}}$, Therefore,
\[
|z| \leq \left(1 + \alpha+\beta + \sum_{j=1}^{n} |q_j|^2 \right)^{1/6}.\]
\end{proof}

\noindent To derive the additional bounds, we define the auxiliary polynomial:
\begin{eqnarray*}
    f_{r_{1}}(z)&=& f_{r}(z)*(q_{n} -z)\\&=& z^{n+1}-z^{n-1}v_{n}-z^{n-2}v_{n-1}-\dots- zv_{2}-v_{1},
\end{eqnarray*}
where $v_{j}=q_{j}q_{n}-q_{j-1}$ for $j=1,2,\dots,n,$ with $q_{0}=0$.
   By Theorem $A_{5}$ $f_{r}(z)*(q_{n} -z)=0$ if and only if either $f_{r}(z)=0$ or  $f_{r}(z)\neq 0$ implies $f_{r}(z)^{-1}zf_{r}(z)- q_{n}=0$. Notice
   that  $f_{r}(z)^{-1}zf_{r}(z)- q_{n}=0$ is equivalent to  $f_{r}(z)^{-1}zf_{r}(z)=q_{n}$ and if $f_{r}(z)\neq 0$ this implies that $z=q_{n}$. So the only zeros of $f_{r}(z)*(q_{n} -z)$ are $z=q_{n}$ and the zeros of $f_{r}(z)$.
   The corresponding companion matrix $Cf_{r_{1}}$ of $f_{r_{1}}$ is given by
   \[
C_{f_{r_{1}}} = 
\begin{pmatrix}
0 & 0 & 0 & \cdots &0& v_{1} \\
1 & 0 & 0 & \cdots &0& v_{2} \\
\vdots & \vdots & \vdots & \ddots & \vdots \\
0 & 0 & 0 & \cdots &0 & v_{n} \\
0 & 0 & 0 & \cdots &1& 0
\end{pmatrix}.
\] 
We have
  \[
C^{2}_{f_{r_{1}}} = 
\begin{pmatrix}
0 & 0 & 0 & \cdots &0& v_{1}&0 \\
0 & 0 & 0 & \cdots &0& v_{2} &v_{1}\\
1 & 0 & 0 & \cdots &0& v_{3} &v_{2}\\
\vdots & \vdots & \vdots & \ddots & \vdots \\
0 & 0 & 0 & \cdots &0 & v_{n} &v_{n-1}\\
0 & 0 & 0 & \cdots &1 & 0 & v_{n}
\end{pmatrix}
\] 
\[
C^{3}_{f_{r_{1}}} = 
\begin{pmatrix}
0 & 0 & 0 & \cdots 0&0& v_{1}&0 &v_{1}v_{n}\\
0 & 0 & 0 & \cdots0 & 0 & v_{2}&v_{1} &v_{2}v_{n}\\
0 & 0 & 0 & \cdots0 & 0 & v_{3}&v_{2} &v_{3}v_{n}+v_{1}\\
1 & 0 & 0 & \cdots0 & 0 & v_{4}&v_{3} &v_{4}v_{n}+v_{2}\\
\vdots & \vdots & \vdots &\ddots & \vdots & \vdots & \vdots & \vdots  \\
0 & 0 & 0 & \cdots0 & 0 & v_{n-1}&v_{n-2} &v_{n-1}v_{n}+v_{n-3}\\
0 & 0 & 0 & \cdots 1 & 0 & v_{n}&v_{n-1} &v^{2}_{n}+v_{n-2}\\
0 & 0 & 0 & \cdots 0 & 1 & 0 & v_{n} &v_{n-1}
\end{pmatrix}.
\]

\begin{theorem}
Let $f_{r_{1}}(z)= z^{n+1}-z^{n-1}v_{n}-z^{n-2}v_{n-1}-\dots- zv_{2}-v_{1}.$ Then for any zero $z$ of $f_{r_{1}}$, we have:
\[
|z| \leq \left(1 + 2 \sum_{j=1}^{n} |v_j|^2 \right)^{1/4}.\]
\end{theorem}

\begin{proof}
Consider the companion matrix  $C_{f_{r_{1}}}^2$ and decompose it as $C_{f_{r_{1}}}^2 = M + N + L$, where:
\[
M = 
\begin{bmatrix}
0 & 0 & \cdots & 0&0 & 0 \\
0 & 0 & \cdots & 0&0& v_{1} \\
0 & 0 & \cdots & 0 &0& v_{2} \\
\vdots  &\vdots & \ddots &\vdots& \vdots & \vdots \\
0 & 0 & \cdots & 0 & 0&v_{n-1}\\
0 & 0 & \cdots & 0 & 0&v_{n}
\end{bmatrix},
\quad
N = 
\begin{bmatrix}
0 & 0 & \cdots & 0&v_{1} & 0 \\
0 & 0 & \cdots & 0&v_{2}& 0 \\
0 & 0 & \cdots & 0 &v_{3}& 0 \\
\vdots  &\vdots & \ddots &\vdots& \vdots & \vdots \\
0 & 0 & \cdots & 0 & v_{n}&0\\
0 & 0 & \cdots & 0 & 0&0
\end{bmatrix},
\quad
L = 
\begin{bmatrix}
0 & 0 \\
I_{n-2} & 0
\end{bmatrix},
\]
where $I_{n-2} $ is identity matrix of order $n-2.$
Then
\[
 MN^{H}=NM^{H}=ML^{H}  = LM^{H} =   NL^{H}  =L N^{H}= 0.
\] Thus by Lemma \ref{lemma 2} and triangle inequality, we have
\begin{eqnarray*}
    \|C_{f_{r_{1}}}^2\|^2_{2} &=&\|(M +N +L)^{H}((M +N +L)\|_{2}\\&= & \|M^{H}M+N^{H}N+L^{H}L\|_{2}\\&\le&  \|M^{H}M\|_{2} + \|N^{H}N\|_{2} + \|L^{H}L\|_{2}\\&=&  2 \sum_{j=1}^{n} |v_j|^2+1.
\end{eqnarray*}
Thus:
\[
\|C_{f_{r_{1}}}^2\|_{2} \leq \left(1 + 2 \sum_{j=1}^{n} |v_j|^2 \right)^{1/2}.
\]

Since $|z| \leq \|C_{f_{r_{1}}}^2\|^{1/2}_{2}$ for any zero $z$ of $C_{f_{r_{1}}}$, we have:
\[
|z| \leq \left(1 + 2 \sum_{j=1}^{n} |v_j|^2 \right)^{1/4}.\]
Hence the result follows.
\end{proof}

\begin{theorem}
Let $f_{r_{1}}(z)= z^{n+1}-z^{n-1}v_{n}-z^{n-2}v_{n-1}-\dots- zv_{2}-v_{1}.$ Then for any zero  $z$ of $f_{r_{1}}$,
\[|z| \leq \left(1 + \gamma + 2 \sum_{j=1}^{n} |v_j|^2 \right)^{1/6},\]
where $\gamma = |v_{1}v_{n}|^2 + |v_{2}v_{n}|^2 + \sum_{j=1}^{n-1} |v_j+v_{n}v_{j+2}|^{2}, ~ v_{n+1}=0 $.
\end{theorem}

\begin{proof}
We consider $C_{f_{r_1}}^3$ and use a decomposition strategy similar to Theorem 2. \\Let
\[
C_{f_{r_{1}}}^3 = M + N + L + G,
\]
\[M = 
\begin{bmatrix}
0 & 0 & \cdots & 0&0 & v_{1}v_{n} \\
0 & 0 & \cdots & 0&0& v_{2} v_{n}\\
0 & 0 & \cdots & 0 &0& v_{3}v_{n} +v_{1}\\
\vdots  &\vdots & \ddots &\vdots& \vdots & \vdots \\
0 & 0 & \cdots & 0 & 0&v^{2}_{n}+v_{n-2}\\
0 & 0 & \cdots & 0 & 0&v_{n-1}
\end{bmatrix},
\quad
N = 
\begin{bmatrix}
0 & 0 & \cdots & 0&0 & 0 \\
0 & 0 & \cdots & 0&v_{1}& 0 \\
0 & 0 & \cdots & 0 &v_{2}& 0 \\
\vdots  &\vdots & \ddots &\vdots& \vdots & \vdots \\
0 & 0 & \cdots & 0 & v_{n-1}&0\\
0 & 0 & \cdots & 0 & v_{n}&0
\end{bmatrix},
\]
\[
L= 
\begin{bmatrix}
0 & 0 & \cdots  &v_{1} & 0 &0\\
0 & 0 & \cdots  &v_{2}& 0&0 \\
0 & 0 & \cdots  &v_{3}& 0&0 \\
\vdots  &\vdots & \ddots &\vdots& \vdots & \vdots \\
0 & 0 & \cdots  & v_{n}&0&0\\
0 & 0 & \cdots  & 0&0&0
\end{bmatrix},
\quad G=\begin{bmatrix}
    0&0\\I_{n-3}&0
\end{bmatrix},\]
where $I_{n-3} $ is identity matrix of order $n-3.$
Then
\[
 MN^{H}=NM^{H}=ML^{H}= LM^{H} = NL^{H}=L N^{H}\]\[=MG^{H}=GM^{H}=NG^{H}=GN^{H} =LG^H=GL^{H}=0.
\]
Thus by Lemma \ref{lemma 1} and triangle inequality, we get
\[
\|C_{f_{r_1}}^3\|^2_{2} \leq 1 + \gamma + 2 \sum_{j=1}^{n} |v_j|^2,
\]
 where $\gamma = |v_{1}v_{n}|^2 + |v_{2}v_{n}|^2 + \sum_{j=1}^{n-1} |v_j+v_{n}v_{j+2}|^{2},  v_{n+1}=0 $\ $C_{p_1}^3$.

Therefore:
\[
\|C_{f_{r_{1}}}^3\|_{2} \leq \left(1 + \gamma + 2 \sum_{j=1}^{n} |v_j|^2 \right)^{1/2}.
\]

Since $|z| \leq \|C_{f_{{r}_1}}^3\|^{1/3}_{2}$ for every zero $z$ of $C_{f_{r_1}}$, Therefore,
\[
|z| \leq \left(1 + \gamma + 2 \sum_{j=1}^{n} |v_j|^2 \right)^{1/6}.\]
\end{proof}

\section{Conclusion}

We have derived new bounds for zeros of quaternion polynomial by carefully addressing the challenges posed by non-commutativity. Our approach leverages the well behaved theory of left eigenvalues and employs matrix inequalities involving spectral norm adapted to quaternion matrices.
These results demonstrate that despite the complexities of quaternion algebra, many classical results from polynomial theory and matrix analysis can be extended through appropriate modifications. The bounds derived here provide practical tools for locating zeros of quaternion polynomials, with potential applications in quaternion signal processing, quantum mechanics, and computer graphics.

\section{Declaration}
\textbf{Availabilty of data and material}\\
 Data availability is not applicable to this article as no new data were created or analyzed in this study.\\

\noindent \textbf{Competing interests}\\
The author declare that they have no competing interests.\\

\noindent \textbf{Funding}\\
The research of first author is supported by DST Inspire Fellowship (ID No. IF210629).


\begin{thebibliography}{999}
        \bibitem{ASL}
S. L. Adler, \textit{Quaternionic Quantum Mechanics and Quantum Fields}, Oxford University Press, New York, (1995).

 \bibitem{ASS}
S. S. Ahmad and I. Ali, \textit{ Bounds for eigenvalues of matrix polynomials over quaternion division algebra}, Advances in Applied Clifford Algebras, \textbf{26} (2016), 1095--1125.
 
\bibitem{ASS2} S. S. Ahmad and I. Ali, \textit{Localization theorems for matrices and bounds for the zeros of polynomials over quaternion division algebra}, Filomat, \textbf{32} (2018), 553--573.

\bibitem{BA}A. Baker, \textit{Right eigenvalues for quaternionic matrices: a topological approach}, Linear Algebra and Its Applications, \textbf{286} (1999), 303--309.

 \bibitem{BJL} J.L. Brenner, \textit{Matrices of quaternions}, Pacific Journal of Mathematics, \textbf{1}, (1951), 329–335.
 
 \bibitem{CJH} J. H. Conway and D. A. Smith, \textit{On Quaternions and Octonions: Their Geometry, Arithmetic, and Symmetry}, A K Peters, Natick, (2002).

\bibitem{DR} I. Dar, N. A. Rather and I. Faiq,\textit{ Bounds on the zeros of quaternionic polynomials using matrix methods}, Filomat, \textbf{38:9} (2024),  3001--3010.
\bibitem{GS} G. Gentili, D. C Struppa, \textit{On the multiplicity of zeros of polynomials with quaternionic coefficients,} Milan Journal of Mathematics, \textbf{76}, 15-25 (2008)
\bibitem{ABG}
A. B. Gerstner, R. Byers and V. Mehrmann, \textit{A quaternion QR algorithm}, Numerische Mathematik, \textbf{55} (1989), 83--95.

\bibitem{GM} B. Gordon and T. S. Motzkin, \textit{On the zeros of polynomials over division rings}, Transactions of the American Mathematical Society, \textbf{116} (1965), 218--226.

\bibitem{HTL} T. L. Hankins, \textit{Sir William Rowan Hamilton}, The Johns Hopkins University Press, Baltimore, (1980).

\bibitem{HL} 
L. Huang and W. So, \textit{On the left eigenvalues of a quaternionic matrix}, Linear Algebra and Its Applications \textbf{332}, (2001), 105-116.

\bibitem{JD} D. Janovsk\'{a} and G. Opfer, \textit{A note on the computation of all zeros of simple quaternionic polynomials}, SIAM Journal on Numerical Analysis, \textbf{48} (2010), 244--256.

\bibitem{JO} D. Janovsk\'{a} and G. Opfer, \textit{The classification and the computation of the zeros of quaternionic, two sided polynomials}, Numerische Mathematik \textbf{115} (2010), 81--100.

\bibitem{JQ} O. Jan and I. Qasim, \textit{On the location of eigenvalues of quaternion matrix polynomials}, Applicable Analysis and Discrete Mathematics, \textbf{19} (2025), 411--421.

\bibitem{BK} B. Kalantari, \textit{Algorithms for quaternion polynomial root finding}, Journal of Complexity, \textbf{29} (2013), 302--322.
\bibitem{KI} M. W. Khadim, I. Ali, \textit{Bounds for the zeros of Quaternionic polynomials}, Gulf Journal of Mathematics, \textbf{22} (2026), 1-16. 
\bibitem{IN} I. Niven, \textit{Equations in Quaternions}, The American Mathematical Monthly, \textbf{48} (1941), 654--661.
\bibitem{O} G. Opfer, \textit{Polynomials and Vandermonde matrices over the field of quaternions}, Electronic Transactions on Numerical Analysis \textbf{36} (2009), 9--16.

\bibitem{PR} 
R. Pereira, \textit{ Quaternionic polynomials and behavioral systems}. Ph.D. thesis, University of Aveiro (2006).

\bibitem{PS} A. Pogorui and M. Shapiro, \textit{ On the structure of the set of zeros of quaternionic polynomials}, Complex Variables and Elliptic Functions, \textbf{49} (2004), 379--388.
\bibitem{IQ} I. Qasim, \textit{ Location of zeros of quaternionic polynomials}, Annali dell'Università di Ferrara \textbf{71} (2025), 14--28.
\bibitem{RL} 
L. Rodman, \textit{Stability of invariant subspaces of quaternion matrices}, Complex Analysis and Operator Theory \textbf{6} (2012), 1069--1119.

\bibitem{SR} 
R. Serodio, E. Pereira and J. Vitória, \textit{Computing the zeros of quaternion polynomials}, Computers \& Mathematics with Applications \textbf{42} (2001), 1229--1237.



\bibitem{WRM} R. M. W. Wood, \textit{Quaternionic eigenvalues}, Bulletin of the London Mathematical Society, \textbf{17} (1985), 137--138.
\bibitem{WRL}  R. L. Wu, \textit{Distribution and estimation for eigenvalues of dual quaternion matrices}, Computers \& Mathematics with Applications \textbf{55} (2008), 1998--2004.


\bibitem{ZF1} 
F. Zhang, \textit{Gershgorin type theorems for quaternionic matrices}, Linear Algebra and Its Applications \textbf{424} (2007), 139--155.


\bibitem{ZF2} 
F. Zhang, \textit{Quaternions and matrices of quaternions}, Linear Algebra and Its Applications \textbf{251} (1997), 21--57.


\bibitem{ZJ} 
L. Zou, Y. Jiang and J. Wu, \textit{Location for the right eigenvalues of quaternion matrices}, Journal of Applied Mathematics and Computing \textbf{38} (2012), 71--83.




\end{thebibliography}
\end{document}